\documentclass[a4paper,11pt]{article}

\usepackage[english]{babel}
\usepackage{amsmath,amsthm,amssymb,thmtools,enumitem,mathrsfs,mathtools,url,doi,titlesec,bm,bbm,graphicx,commath}
\usepackage[left=2.5cm,right=2.5cm,top=2.5cm,bottom=2.5cm,bindingoffset=0cm]{geometry}
\usepackage[titletoc,title]{appendix}

\setlist[enumerate]{topsep=0pt,label=\textup{(\arabic*)},leftmargin=\parindent,labelsep=.5em}
\setlist{noitemsep}

\usepackage{quoting}
\quotingsetup{vskip=0pt}

\usepackage[T1]{fontenc}

\usepackage{tikz}
\usetikzlibrary{arrows,matrix}
\usepackage{tikz-cd}
\tikzset{
    labl/.style={anchor=south, rotate=90, inner sep=.5mm}
}
\tikzset{
  symbol/.style={
    draw=none,
    every to/.append style={
      edge node={node [sloped, allow upside down, auto=false]{$#1$}}}
  }
}

\theoremstyle{plain}
\declaretheorem[numberlike=subsection]{proposition}
\declaretheorem[numberlike=subsection]{theorem}

\declaretheorem[numberlike=subsection]{lemma}

\declaretheorem[numbered=no,name=Theorem]{IntroThm}

\theoremstyle{definition}

\declaretheorem[numberlike=subsection]{example}
\declaretheorem[numberlike=subsection]{remark}

\declaretheorem[numbered=no,name=Acknowledgement]{acknowl}

\numberwithin{equation}{subsection}

\titleformat{\section}[block]
  {\filcenter\normalfont\large\bfseries}{\thesection.}{.5em}{}
\titleformat{\subsection}[runin]
  {\normalfont\bfseries}{\thesubsection}{.5em}{}

\titlespacing*{\section}{0pt}{6ex plus 1ex minus .2ex}{3ex plus .2ex}
\titlespacing*{\subsection}{0pt}{\topsep}{.5em}

\input{BMmacsLatex.sty}

\begin{document}

\begin{center}
\textbf{\Large COMPUTING DISCRETE INVARIANTS OF}
\vspace{2mm}

\textbf{\Large VARIETIES IN POSITIVE CHARACTERISTIC}
\vspace{4mm}

{\Large I.~\textit{Ekedahl-Oort types of curves}}
\vspace{8mm}

\textit{by}
\bigskip

{\Large Ben Moonen}
\end{center}
\vspace{6mm}

{\small 

\noindent
\begin{quoting}
\textbf{Abstract.} We develop a method to compute the Ekedahl--Oort type of a curve~$C$ over a field~$k$ of characteristic~$p$ (which is the isomorphism type of the $p$-kernel group scheme~$J[p]$, where $J$ is the Jacobian of~$C$). Part of our method is general, in that we introduce the new notion of a Hasse--Witt triple, which re-encodes in a useful way the information contained in the Dieudonn\'e module of~$J[p]$. For complete intersection curves we then give a simple method to compute this Hasse--Witt triple. An implementation of this method is available in Magma.
\medskip

\medskip

\end{quoting}

} 
\vspace{6mm}

\section{Introduction}

An elliptic curve $E$ over a field~$k$ of characteristic~$p$ can be either ordinary or supersingular. If $E$ is given by a homogeneous cubic equation $f=0$ then $E$ is supersingular if and only if the coefficient of $(X_0 X_1 X_2)^{p-1}$ in~$f^{p-1}$ is non-zero. The goal of the present paper (and its sequel~\cite{CDI2}) is to generalize this, first to more general curves, and later to certain varieties of higher dimension.

For the rest of this introduction, assume the ground field~$k$ is algebraically closed; let $\sigma$ be its Frobenius automorphism. For curves~$C$ of genus~$g$ there are $2^g$ possibilities for the isomorphism class of the group scheme~$J[p]$, the $p$-kernel of the Jacobian of~$C$. This isomorphism class is often referred to as the Ekedahl--Oort type of~$C$. The Ekedahl--Oort stratification on the moduli space~$\cA_g$ of $g$-dimensional abelian varieties (say with principal polarization) has been studied in great detail, and the underlying theory has been refined so as to be able to handle more general Shimura varieties. It appears, however, that not so many results concerning explicit calculations are as yet available. The main question that we answer in this paper is how, for a complete intersection curve $C \subset \mP^n$ given by homogeneous equations $f_1 = \cdots = f_{n-1} = 0$ and $p$ not dividing the degrees of these equations, we can calculate the Ekedahl--Oort type of~$C$. (For hyperelliptic curves, which are more amenable to explicit calculation, this question was answered in~\cite{DevaHalli}.)

At the heart of the paper lies a new approach to such questions, which in the second part of this work will be extended to varieties of higher dimensions, or better: to cohomology in degree~$>1$. The Ekedahl--Oort type of a curve~$C$ can be calculated from the Dieudonn\'e module of~$J[p]$, which is just the first de Rham cohomology group $H^1_\dR(C/k)$, equipped with semi-linear operations $F$ and~$V$. Thinking of $H^1_\dR$ as being an extension of $Q = H^1(C,\cO_C)$ by its dual $Q^\vee \cong H^0(C,\Omega^1_C)$, we realize that the induced action of~$F$ on~$Q$ (which is just the classical Hasse--Witt operator) is relatively easy to calculate. Unlike the case of elliptic curves, however, for $g\geq 2$ knowing the Hasse--Witt operator $\Phi \colon Q \to Q$ is in general not enough to determine the Ekedahl--Oort type of~$C$. The idea that we pursue is that there is not too much information missing. We make this precise in the notion of a Hasse--Witt triple.

By definition, a Hasse--Witt triple $(Q,\Phi,\Psi)$ consists of a finite dimensional $k$-vector space~$Q$ equipped with a $\sigma$-linear endomorphism $\Phi \colon Q \to Q$  and a $\sigma$-linear bijective map $\Psi\colon \Ker(\Phi) \to \Coker(\Phi)^\vee$. To a Dieudonn\'e module of a $p$-kernel group scheme~$G$ (more precisely: a polarized $p$-kernel of a $p$-divisible group) one can associate a Hasse--Witt triple. There is also an easy way to go back from a HW-triple to a Dieudonn\'e module. In Theorem~\ref{thm:DM=HW} we show that this gives a bijective correspondence between isomorphism classes of Dieudonn\'e modules and isomorphism classes of HW-triples.

The main point of the paper, then, is that the Hasse--Witt triple associated to a curve can be efficiently computed, at least for complete intersection curves if $p$ does not divide the degrees of the equations. (We expect that something similar can be done in case the curve~$C$ is given to us in a different way, e.g., as a branched cover of~$\mP^1$.) Let us describe the main result in the case of a plane curve $C \subset \mP^2$. The general case, which is given in Theorem~\ref{thm:HWCI}, is entirely similar but is notationally a little more involved.

\begin{IntroThm}
Let $C = \cZ(f) \subset \mP^2$ be a plane curve of degree~$d \geq 3$ over an algebraically closed field~$k$ of characteristic~$p$ with $p\nmid d$. Let $\sS = k[X_0,X_1,X_2]$, and define $\sT = k[X_0^{\pm 1},X_1^{\pm 1},X_2^{\pm 1}]/L$, where $L$ is the $k$-linear span of all monomials $X_0^{e_0} X_1^{e_1} X_2^{e_2}$ for which at least one of the exponents~$e_i$ is non-negative. Let $\sS = \oplus_{m\geq 0}\, \sS_m$ and $\sT = \oplus_{m\leq -3}\, \sT_m$ be the natural gradings. Define
\[
Q = \sT_{-d}\, ,\quad \text{and}\quad Q^\prime = \bigl\{\xi \in \sT_{-2d} \bigm| \tfrac{\partial f}{\partial X_j} \cdot \xi = 0\ \text{in $\sT_{-d-1}$, for all $j=0,1,2$}\bigr\}\, .
\]
The space $\sU = \bigl\{\xi \in \sT_{-3d+3} \bigm| \tfrac{\partial f}{\partial X_j} \cdot \xi = 0\quad \text{for all $j=0,1,2$}\bigr\}$ is $1$-dimensional. Choose $0 \neq \sansu \in \sU$. Then $\sS_{d-3} \isomarrow Q^\prime$ by $g \mapsto g\cdot \sansu$, and the bilinear map $Q \times Q^\prime \to k$ that sends $(q,g\cdot \sansu)$ to the coefficient of~$X_0^{-1} X_1^{-1} X_2^{-1}$ in $g \cdot q$ is a perfect pairing. Let $\theta \colon Q^\prime \isomarrow Q^\vee$ be the associated isomorphism.

Define
\[
\Phi \colon Q \to Q\qquad \text{by}\quad \Phi[A] = \bigl[f^{p-1} \cdot A^p\bigr]
\]
and define
\[
\Psi \colon \Ker(\Phi) \to Q^\vee\qquad \text{by}\quad \Psi[A] = \theta \bigl[f^{p-2} \cdot A^p\bigr]\, .
\]
Then $(Q,\Phi,\Psi)$ is a Hasse--Witt triple whose associated Dieudonn\'e module is isomorphic to the Dieudonn\'e module of the group scheme~$J[p]$, where $J$ is the Jacobian of~$C$.
\end{IntroThm}

In conclusion, to compute the Ekedahl--Oort type of~$C$ it suffices to enhance the pair consisting of the space $H^1(C,\cO_C)$ with its Hasse--Witt operator~$\Phi$ to a Hasse--Witt triple, and for the new ingredient, the operator~$\Psi$, we have a formula which is just as elegant as, and in fact very similar to, the classical formula for~$\Phi$.

\begin{acknowl}
My sincerest thanks go to Wieb Bosma, for all his help with the Magma-implementation of the above theorem; see Section~\ref{sec:Magma}. I should also like to thank the referees for helpful comments and suggestions, and one referee in particular for providing a cleaned-up version of the Magma code, which is available on my webpage.
\end{acknowl}

\section{Dieudonn\'e modules and Hasse--Witt triples}
\label{sec:DM}

\subsection{}
\label{ssec:DM1}
Let $k$ be a perfect field of characteristic $p>0$ with Frobenius automorphism $\sigma \colon k \to k$. We are interested in Dieudonn\'e modules $(M,F,V)$, where:
\begin{enumerate}
\item $M$ is a finite dimensional $k$-vector space;
\item $F \colon M \to M$ is $\sigma$-linear and $V \colon M \to M$ is $\sigma^{-1}$-linear;
\item $\Ker(F) = \Image(V)$ and $\Ker(V) = \Image(F)$.
\end{enumerate}
These are precisely the Dieudonn\'e modules associated with $p$-kernels of $p$-divisible groups. The term \emph{Dieudonn\'e module} will henceforth refer to triples $(M,F,V)$ satisfying (1)--(3). We usually denote such a Dieudonn\'e module by the single letter~$M$, leaving $F$ and~$V$ implicit.

Let $M$ be a Dieudonn\'e module over~$k$. By a \emph{polarization} of~$M$ we mean a non-degenerate alternating bilinear form $b \colon M \times M \to k$ such that
\begin{equation}\label{eq:bFxy}
b\bigl(F(x),y\bigr) = b\bigl(x,V(y)\bigr)^p \qquad \text{for all $x$, $y \in M$.}
\end{equation}
If such a polarization exists, the dimension of~$M$ is even, say $\dim_k(M) = 2g$, and the subspaces $\Ker(F) \subset M$ and $\Ker(V) \subset M$ are maximal isotropic.

If $(M,F,V,b)$ is a polarized Dieudonn\'e module, the Verschiebung~$V$ can be recovered from the remaining ingredients by~\eqref{eq:bFxy}.

\subsection{}\label{ssec:classif}
Fix an integer $g \geq 1$ and let $(W,S)$ denote the Weyl group of the reductive group $\Sp_{2g}$. Concretely,
\[
W = \bigl\{ \pi \in \gS_{2g} \bigm| \pi(i) + \pi(2g+1-i) = 2g+1\quad \text{for all $i \in \{1,\ldots,2g\}$} \bigr\}
\]
($\gS_n$ = symmetric group on $n$ leters), and $S = \{s_1,\ldots,s_g\}$ with $s_i = (i\quad i+1)\, (2g-i\quad 2g-i+1)$ for $i<g$ and $s_g = (g\quad g+1)$. Let $X = S\setminus\{s_g\}$, and let $W_X \subset W$ be the subgroup generated by~$X$. Then $W_X \cong \gS_g$ and $W_X\backslash W$ is a set of $2^g$ elements. For every class in $[w] \in W_X\backslash W$ there is a unique representative $\dot{w} \in W$ of minimal length.

To a polarized Dieudonn\'e module~$M$ with $\dim_k(M) = 2g$ we can associate an element $w(M) \in W_X\backslash W$. Here we only give a quick summary of how this is done; for details see \cite{GSAS} and~\cite{OortE33}. The first step is to build a symplectic flag
\[
\cD_\bullet\; :\qquad (0) = \cD_0 \subsetneq \cD_1 \subsetneq \cdots \subsetneq \cD_r = M\, ,
\]
in the following way: We start with $(0) \subset M$. Then we apply the operations~$F$ (taking the image under~$F$) and $V^{-1}$ (taking the pre-image under~$V$) to each term, which yields $(0) \subset V^{-1}(0) = FM \subset M$. This procedure we iterate. After finitely many iterations the process stabilizes; define $\cD_\bullet$ to be the flag in~$M$ that is obtained. Put differently, $\cD_\bullet$ is the coarsest flag that is stable under $F$ and~$V^{-1}$, in the sense that for any term~$\cD_i$ there are indices $a$ and~$b$ such that $F\cD_i = \cD_a$ and $V^{-1}\cD_i = \cD_b$.

This ``canonical flag''~$\cD_\bullet$ is not, in general, a full flag in~$M$. Let $\tilde{\cD}_\bullet$ be a refinement of~$\cD_\bullet$ to a full symplectic flag. Then the relative position of the flags $(0) \subset \Ker(F) \subset M$ and~$\tilde{\cD}_\bullet$ is an element
\[
\relpos\bigl(\Ker(F),\tilde{\cD}_\bullet\bigr) \in W_X\backslash W
\]
(see \cite{GSAS}, Section~3) that turns out to be independent of how we choose the refinement~$\tilde{\cD}_\bullet$. Now define $w(M) = \relpos\bigl(\Ker(F),\tilde{\cD}_\bullet\bigr)$.

The following result is based on a classification result for Dieudonn\'e modules that was proven by Kraft~\cite{Kraft} (unpublished) and was later re-obtained by Ekedahl and Oort, see~\cite{OortE33}. The result as we state it can also be found (stated using a different encoding) in \cite{OortTexel}, Section~9; it is a special case of the results in~\cite{GSAS}.

\begin{theorem}
\label{thm:ClassifBT1}
Let $k$ be an algebraically closed field of characteristic~$p$. Then $M \mapsto w(M)$ gives a bijection
\[
\Biggl\{
\vcenter{
\setbox0=\hbox{isomorphism classes of polarized Dieudonn\'e modules}
\copy0\hbox to\wd0{\hfil $(M,F,V,b)$ over~$k$ with $\dim_k(M) = 2g$\hfil}}\Biggr\} \isomarrow W_X\backslash W\, .
\]
\end{theorem}

\begin{remark}
While from a theoretical perspective the encoding in terms of Weyl group cosets is the most natural choice, there are other ways to encode the isomorphism class of a given polarized Dieudonn\'e module. For a nice discussion, see \cite{PriesUlmer}, Section~3. Let us note that it is very easy to read off from~$w(M)$ more basic invariants of~$M$ such as the $p$-rank or the $a$-number. For instance, if $\dot{w} \in W$ is the (unique) representative of~$w(M)$ of minimal length then the $p$-rank of~$M$ equals the number of indices $i \in \{1,\ldots,g\}$ such that $\dot{w}(i) = i+g$, and the $a$-number is given by the number of indices $i \in \{1,\ldots,g\}$ such that $\dot{w}(i) \in \{1,\ldots,g\}$.
\end{remark}

\subsection{}\label{ssec:DMtoHW}
In the next section we will describe a method to compute, for certain classes of curves~$C$ over a field~$k$ of characteristic~$p$, the isomorphism class of the group scheme $J[p]$ over~$\kbar$, where $J$ is the Jacobian of~$C$. This isomorphism class is referred to as the Ekedahl--Oort type of~$C$. Our method is based on the observation that the Dieudonn\'e module of~$J[p]$ can be reconstructed from some data that are easier to compute, as we shall now explain.

As before, let $k$ be a perfect field of characteristic $p>0$. By a \emph{HW-triple over~$k$} (short for ``Hasse--Witt triple'') we mean a triple $(Q,\Phi,\Psi)$ where
\begin{enumerate}
\item $Q$ is a finite dimensional $k$-vector space;
\item $\Phi \colon Q \to Q$ is a $\sigma$-linear map;
\item $\Psi \colon \Ker(\Phi) \isomarrow \Image(\Phi)^\perp$ is a $\sigma$-linear bijective map.
\end{enumerate}
Here $\Image(\Phi)^\perp \subset Q^\vee$ is the subspace given by
\[
\Image(\Phi)^\perp = \bigl\{\lambda \in Q^\vee \bigm| \lambda(q) = 0\quad \text{for all $q \in \Image(\Phi)$} \bigr\}\, .
\]
(Instead of $\Image(\Phi)^\perp$ we could also write $\Coker(\Phi)^\vee$.) The name alludes to the fact that, in the example of a curve~$C$ over~$k$, the map~$\Phi$ that we will consider is the Hasse--Witt operator on the space $Q = H^1(C,\cO_C)$. Note that the spaces $\Ker(\Phi)$ and $\Image(\Phi)^\perp$ have the same dimension, so to check bijectivity of a $\sigma$-linear map $\Psi \colon \Ker(\Phi) \to \Image(\Phi)^\perp$, it suffices to establish either its injectivity or surjectivity.

To a polarized Dieudonn\'e module $(M,F,V,b)$ we can associate a HW-triple by taking $Q = M/\Ker(F)$ with $\Phi$ given by the composition
\[
M/\Ker(F) \xrightarrow{~F~} M \twoheadrightarrow M/\Ker(F)\, .
\]
For $x \in \Ker(\Phi)$ we have $F(x) \in \Ker(F) \subset M$, and because $\Ker(F)$ is a maximal isotropic subspace of~$M$ the linear map $b\bigl(-,F(x)\bigr) \colon M \to k$ is an element of $Q^\vee = \bigl(M/\Ker(F)\bigr)^\vee \subset M^\vee$. Define $\Psi \colon \Ker(\Phi) \to Q^\vee$ by $\Psi(x) = b\bigl(-,F(x)\bigr)$. Note that $\Psi$ is injective, because $b\bigl(-,F(x)\bigr) = 0$ implies that $F(x) = 0$ and $F$ is injective on $Q = M/\Ker(F)$. Further, for $x \in \Ker(\Phi)$ and $y \in Q$ we have $b\bigl(\Phi(y),F(x)\bigr) = 0$ because $\Phi(y) \in Q$ is the class of $F(y)$ modulo~$\Ker(F)$ and $\Image(F) \subset M$ is an isotropic subspace. This shows that $\Psi$ takes values in~$\Image(\Phi)^\perp$. As remarked above, the injectivity of~$\Psi$ implies that it is bijective; so $(Q,\Phi,\Psi)$ is indeed a HW-triple.

\subsection{}\label{ssec:HWtoDM}
In the opposite direction, we can associate to a HW-triple $(Q,\Phi,\Psi)$ a polarized Dieudonn\'e module, well-determined up to isomorphism. This works as follows. Define $M = Q \oplus Q^\vee$, and let $b \colon M \times M \to k$ be the form given by $b\bigl((q,\lambda),(q^\prime,\lambda^\prime)\bigr) = \lambda^\prime(q) - \lambda(q^\prime)$.

Write $R_1 = \Ker(\Phi) \subset Q$ and choose a subspace $R_0 \subset Q$ which is a complement of~$R_1$. Then we can define
\[
F \colon M = R_0 \oplus R_1 \oplus Q^\vee \longrightarrow M = Q \oplus Q^\vee
\]
by $F(r_0,r_1,\lambda) = \bigl(\Phi(r_0),\Psi(r_1)\bigr)$, and we let $V$ be the unique map such that \eqref{eq:bFxy} is satisfied. To describe~$V$ more explicitly, note that the maps
\[
\Psi \colon R_1 \isomarrow \Image(\Phi)^\perp = \bigl(Q/\Image(\Phi)\bigr)^\vee \subset Q^\vee
\qquad\text{and}\qquad
\Phi \colon R_0 \isomarrow \Image(\Phi) \subset Q
\]
dualize to $\sigma^{-1}$-linear maps $\Psi^\vee \colon Q/\Image(\Phi) \isomarrow R_1^\vee$ and $\Phi^\vee \colon Q^\vee/\Image(\Phi)^\perp = \Image(\Phi)^\vee \isomarrow R_0^\vee$; now
\[
V \colon M = Q \oplus Q^\vee \longrightarrow M = Q \oplus R_0^\vee \oplus R_1^\vee
\]
is given by
\[
V(q,\lambda) = \bigl(0,\Phi^\vee(\lambda \bmod \Image(\Phi)^\perp),-\Psi^\vee(q \bmod \Image(\Phi))\bigr)\, .
\]
As it is clear that $\Image(F) = \Image(\Phi) \oplus \Image(\Phi)^\perp = \Ker(V)$ and $\Ker(F) = (0) \oplus Q^\vee = \Image(V)$, it follows that $(M,F,V,b)$ is a polarized Dieudonn\'e module.

\begin{lemma}\label{lem:IndepQ1}
In the construction of \emph{\ref{ssec:HWtoDM}}, the isomorphism class of the polarized Dieudonn\'e module $(M,F,V,b)$ is independent of the choice of $R_0 \subset Q$.
\end{lemma}

\begin{proof}
Let $R_0^\prime \subset Q$ be another complement of~$R_1$, and let $F^\prime \colon M \to M$ and $V^\prime \colon M \to M$ be the associated Frobenius and Verschiebung on~$M$. There exists a linear map $t \colon R_0 \to R_1$ such that $R_0^\prime \subset Q = R_0 \oplus R_1$ is the graph of~$t$. Let $m = (q,\lambda)$ be an element of~$M$. If $q = (r_0,r_1) \in R_0 \oplus R_1$ then $q = \bigl(r_0,t(r_0)\bigr) + \bigl(0,r_1-t(r_0)\bigr)$ with $\bigl(r_0,t(r_0)\bigr) \in R_0^\prime$ and $\bigl(0,r_1-t(r_0)\bigr) \in R_1$. Because $R_1 = \Ker(\Phi)$ we find that
\[
F(m) = \bigl(\Phi(r_0),\Psi(r_1)\bigr)
\qquad\text{and}\qquad
F^\prime(m) = \bigl(\Phi(r_0),\Psi(r_1- t(r_0))\bigr)\, .
\]
The main point is now to show that there exists a self-dual linear map $u \colon Q \to Q^\vee$ such that $u \circ \Phi = - \Psi \circ t$ as maps from~$R_0$ to~$Q^\vee$. Since the map $\Phi \colon R_0 \to Q$ is injective and $\Image(\Phi) = \Phi(R_0)$, we can define a linear map $u_0 \colon \Image(\Phi) \to \Image(\Phi)^\perp \subset Q^\vee$ by the rule $u_0\bigl(\Phi(r_0)\bigr) = - \Psi\bigl(t(r_0)\bigr)$. As the domain of~$u_0$ is~$\Image(\Phi)$ and $u_0$ takes values in $\Image(\Phi)^\perp$, we can find a self-dual $u \colon Q \to Q^\vee$ that extends~$u_0$. Now define $\alpha \colon M \to M$ by $\alpha(q,\lambda) = (q,\lambda+u(q))$. Then one readily verifies that $\alpha$ is a symplectic automorphism of~$(M,b)$ with $\alpha \circ F = F^\prime \circ \alpha$, and hence also $\alpha \circ V = V^\prime \circ \alpha$.
\end{proof}

\begin{theorem}\label{thm:DM=HW}
Let $k$ be a perfect field of characteristic~$p$. Then the construction given in \emph{\ref{ssec:DMtoHW}} gives a bijection
\[
\Biggl\{
\vcenter{
\setbox0=\hbox{isomorphism classes of polarized Dieudonn\'e}
\copy0\hbox to\wd0{\hfil modules $(M,F,V,b)$ with $\dim_k(M) = 2g$\hfil}}\Biggr\}
\isomarrow
\Biggl\{
\vcenter{
\setbox0=\hbox{isomorphism classes of HW-triples}
\copy0\hbox to\wd0{\hfil $(Q,\Phi,\Psi)$ over~$k$ with $\dim_k(Q) = g$\hfil}}\Biggr\}
\]
whose inverse is given by the construction in \emph{\ref{ssec:HWtoDM}}.
\end{theorem}

\begin{proof}
Let $(Q,\Phi,\Psi)$ be a HW-triple, choose a $R_0 \subset Q$ which is a complement of $R_1 = \Ker(\Phi)$, and define $(M,F,V,b)$ as in~\ref{ssec:HWtoDM}. Now it is a routine verification that if we apply the construction of~\ref{ssec:DMtoHW} to $(M,F,V,b)$ we recover the original HW-triple.

In the opposite direction, start with a polarized Dieudonn\'e module $(M,F_M,V_M,b_M)$, and let $(Q,\Phi,\Psi)$ be the associated HW-triple. Define $N = Q\oplus Q^\vee$, and equip it with the alternating bilinear form~$b_N$ given by $b_N\bigl((q,\lambda),(q^\prime,\lambda^\prime)\bigr) = \lambda^\prime(q) - \lambda(q^\prime)$. Choose a complement $R_0 \subset Q$ of $R_1 = \Ker(\Phi)$, and let $F_N$ and~$V_N$ be the associated Frobenius and Verschiebung on~$N$, as in~\ref{ssec:HWtoDM}.

Let $\pr \colon M \to Q$ be the projection map, and choose a totally isotropic subspace $L \subset M$ (with respect to the form~$b_M$) such that $\Ker(F_M) \cap L = (0)$. The composition $L \hookrightarrow M \xrightarrow{\pr} Q$ is an isomorphism; write $j \colon Q \isomarrow L$ for the inverse. We then obtain an isometry $h\colon N \isomarrow M$ by sending $(q,\lambda) \in Q \oplus Q^\vee$ to $x_\lambda - j(q)$, where $x_\lambda \in \Ker(F_M)$ is the unique element with $\lambda(y) = b_M\bigl(x_\lambda,j(y)\bigr)$ for all $y \in Q$.

Let $F_Q \colon Q \hookrightarrow M$ be the map induced by~$F_M$. Note that $F_Q(R_0) \cap \Ker(F_M) = (0)$, and because $F_Q(R_0) \subset \Image(F_M)$ is an isotropic subspace of~$M$, we may choose the maximal isotropic $L \subset M$ as above in such a way that $F_Q(R_0) \subset L$. We are done if we can show that with this choice of~$L$ we have $F_M \circ h = h \circ F_N$. For this, let $(q,\lambda)$ be an element of $N = Q \oplus Q^\vee$, and write $q = r_0 + r_1$. Then $(F_M\circ h)\bigl(q,\lambda\bigr) = -F_M\bigl(j(q)\bigr) = -F_M\bigl(j(r_0)\bigr) - F_M\bigl(j(r_1)\bigr)$. On the other hand, $F_N(q,\lambda) = \bigl(\Phi(r_0),\Psi(r_1)\bigr)$, so that
\[
(h\circ F_N)\bigl(q,\lambda\bigr) = x_{\Psi(r_1)} - j\bigl(\Phi(r_0)\bigr)\, ,
\]
where $x_{\Psi(r_1)} \in \Ker(F_M)$ is the unique element with $\Psi(r_1)\bigl(y\bigr) = b\bigl(x_{\Psi(r_1)},j(y)\bigr)$ for all $y \in Q$. But if $\tilde{y} \in M$ is any representative of~$y$ then we have $\Psi(r_1)\bigl(y\bigr) = b\bigl(\tilde{y},F_Q(r_1)\bigr)$ by definition of~$\Psi$, and since we may choose $\tilde{y} = j(y)$ it follows that $x_{\Psi(r_1)} = -F_Q(r_1) = -F_M\bigl(j(r_1)\bigr)$. (Note that indeed $-F_Q(r_1) \in \Ker(F_M)$ because $r_1 \in R_1 = \Ker(\Phi)$.) On the other hand, $F_Q(r_0)$ and $j\bigl(\Phi(r_0)\bigr)$ are two elements of~$L$ that under the projection to~$Q$ both map to $\Phi(r_0)$; hence $j\bigl(\Phi(r_0)\bigr) = F_Q(r_0)$, which is the same as $F_M\bigl(j(r_0)\bigr)$. This shows that indeed $F_M \circ h = h \circ F_N$.
\end{proof}

\subsection{}\label{ssec:equiv}
There is an equivalence relation on HW-triples that is weaker than isomorphy and that will be convenient for us to use. Namely, if $(Q_1,\Phi_1,\Psi_1)$ and $(Q_2,\Phi_2,\Psi_2)$ are two HW-triples, we call a $k$-linear bijective map $f \colon Q_1 \to Q_2$ an \emph{equivalence} if there exist constants $c$, $d \in k^*$ such that $\Phi_2 \circ f = c\cdot (f \circ \Phi_1)$ and $\Psi_2 \circ f = d\cdot (f^{\vee,-1} \circ \Psi_1)$.

\begin{proposition}\label{prop:EquivHW}
Let $k$ be an algebraically closed field of characteristic $p > 0$. If two HW-triples over~$k$ are equivalent then they are in fact isomorphic.
\end{proposition}

Note that the assumption that $k$ is algebraically closed cannot be dropped. The reader is invited to skip the proof upon first reading.

\begin{proof}
Let $(Q,\Phi,\Psi)$ be a HW-triple over~$k$, and choose a decomposition $Q = R_0 \oplus R_1$ with $R_1 = \Ker(\Phi)$. For $c$, $d \in k^*$, let $M_{c,d}$ be the polarized Dieudonn\'e module associated, via the construction in~\ref{ssec:HWtoDM}, with the HW-triple $(Q,c\cdot \Phi,d\cdot \Psi)$. By Theorem~\ref{thm:DM=HW}, it suffices to show that $M_{c,d} \cong M_{1,1}$. The proof of this is based on the results of Kraft in~\cite{Kraft}. As a first step we recall how to construct certain ``standard Dieudonn\'e modules''. Start with a natural number~$q$, a map $X \colon (\mZ/q\mZ) \to \{F,V\}$, and a $q$-tuple $a = (a_i)$ in $(k^*)^{\mZ/q\mZ}$. Let $\{e_i\}_{i\in \mZ/q\mZ}$ be the standard basis of $k^{\mZ/q\mZ}$. We give $M_{X,a} := k^{\mZ/q\mZ}$ the structure of a Dieudonn\'e module by declaring that $F(e_i) = a_i \cdot e_{i+1}$ if $X(i) = F$ and $V(a_i \cdot e_{i+1}) = e_i$ if $X(i) = V$. In combination with the requirement that $\Ker(F) = \Image(V)$ and $\Ker(V) = \Image(F)$, these rules uniquely determine the structure of a Dieudonn\'e module; for instance, if $X(i) = V$ then $V(a_i\cdot e_{i+1}) = e_i$ forces $F(e_i) = 0$. If $a_i = 1$ for all $i \in \mZ/q\mZ$ then we write $M_X$ instead of~$M_{X,a}$. Note that $M_{X,a} \cong M_X$ for all $a \in (k^*)^{\mZ/d\mZ}$; an isomorphism $M_X \isomarrow M_{X,a}$ is obtained by $e_i \mapsto c_i \cdot e_i$ for suitable constants~$c_i$. We will need the more general form~$M_{X,a}$ later.

As we shall discuss next, every polarized Dieudonn\'e module can be described in terms of the following two basic cases:

\emph{Case 1:} In the above construction, suppose $q=2r$ is even and the function $X \colon (\mZ/q\mZ) \to \{F,V\}$ has the property that $X(i+r) = F \iff X(i) = V$. Choose $\zeta \in k^*$ with $\zeta^{p^r} = -1$. Then the alternating bilinear form $b \colon M_X \times M_X \to k$ defined by the property that $b(e_i,e_j) = 0$ if $j\neq i+r$ and $b(e_i,e_{i+r}) = \zeta^{p^i}$ is a polarization of~$M_X$.

\emph{Case 2:} In the above construction, define $\check{X} \colon (\mZ/q\mZ) \to \{F,V\}$ by the rule $\check{X}(i) = F \iff X(i) = V$. Let $\{\check{e}_i\}_{i\in \mZ/q\mZ}$ denote the standard basis of~$M_{\check{X}}$. Then the form $b \colon (M_X\oplus M_{\check{X}}) \times (M_X\oplus M_{\check{X}}) \to k$ determined by the rules
\[
b(e_i,e_j) = 0 = b(\check{e}_i,\check{e}_j)\, ,\qquad b(e_i,\check{e}_j) = \delta_{ij}\, ,\qquad b(\check{e}_i,e_j) = -\delta_{ij}
\]
is a polarization of the Dieudonn\'e module $M_X\oplus M_{\check{X}}$.

We will make use of the following facts:
\begin{enumerate}
\item Every Dieudonn\'e module over~$k$ is isomorphic to a direct sum of Dieudonn\'e modules of the form~$M_X$.
This result is due to Kraft~\cite{Kraft}.
\item If $N$ is a Dieudonn\'e module over~$k$ that admits a polarization $b \colon N \times N \to k$ then this polarization is unique up to isomorphism. In other words: if $b^\prime$ is another polarization of~$N$ then $(N,b) \cong (N,b^\prime)$. This was proven by Oort; see \cite{OortTexel}, Theorem~9.4.
\item Every polarized Dieudonn\'e module is isomorphic to a direct sum of terms $(M_X,b)$ as in Case~1 and $(M_X \oplus M_{\check{X}},b)$ as in Case~2. (This follows without much difficulty from the previous two facts.)
\end{enumerate}

To conclude the proof we may now assume that $(Q,\Phi,\Psi)$ is the HW-triple associated with a polarized Dieudonn\'e module $(M_X,b)$ as in Case~1 or a polarized Dieudonn\'e module $(M_X \oplus M_{\check{X}},b)$ as in Case~2. We first treat Case~1. Let $c$, $d \in k^*$ be given. We define a sequence $(a_i)_{i\in \mZ/2r\mZ}$ as follows. If $X(i) = X(i+1) = F$, let $a_i = c$. If $X(i) = F$ and $X(i+1) = V$, let $a_i = d$. If $X(i) = V$ then $X(i+r) = F$, so that $a_{i+r}$ has already been defined; in this case let $a_i = a_{i+r}^{-1}$. One checks that the same form~$b$ as in Case~1 above defines a polarization on~$M_{X,a}$ and that the associated HW-triple is $(Q,c\cdot \Phi,d\cdot \Psi)$.

In Case~2 the calculation is similar but we need to introduce some more notation. Let $I = (\mZ/q\mZ) \coprod (\mZ/q\mZ)$ and let $i\mapsto \bar{i}$ be the involution of~$I$ that exchanges the two copies of $\mZ/q\mZ$. Define a sequence $(a_i)_{i\in I}$ by starting with the rule that $a_i = c$ if $X(i) = X(i+1) = F$ and $a_i = d$ if $X(i) = F$ and $X(i+1) = V$. If $X(i) = V$ then $X(\bar{i}) = F$ so that $a_{\bar{i}}$ has already been defined, and we take $a_i := (a_{\bar{i}})^{-1}$. Again one checks by direct calculation that the same form~$b$ as in Case~2 above defines a polarization on~$M_{X,a} \oplus M_{\check{X},a}$ and that the associated HW-triple is $(Q,c\cdot \Phi,d\cdot \Psi)$. This completes the proof.
\end{proof}

\section{Complete intersection curves}
\label{sec:ComplInters}

In this section we describe a method to calculate the equivalence class of the Hasse--Witt triple associated with a complete intersection curve, assuming the characteristic of our field does not divide the degrees of the equations. When combined with the results of the previous section, this gives a way to compute the Ekedahl--Oort type of such curves.

All curves we consider are assumed to be complete and non-singular.

\subsection{}\label{ssec:DMofJ[p]}
Let $k$ be an algebraically closed field of characteristic $p>0$ with Frobenius automorphism $\sigma \colon k \to k$. Let $C/k$ be a curve of genus $g\geq 1$, and let $J$ be the Jacobian of~$C$. Write $M = H^1_\dR(C/k)$. As is well-known (see \cite{OdaDR}, Section~5) we have a $\sigma$-linear Frobenius $F \colon M\to M$ and a $\sigma^{-1}$-linear Verschiebung $V \colon M \to M$ such that $(M,F,V)$ is the (contravariant) Dieudonn\'e module of the $p$-kernel group scheme~$J[p]$. The principal polarization of~$J$ gives rise to an alternating bilinear form $b \colon M \times M \to k$ satisfying~\eqref{eq:bFxy}, so that $(M,F,V,b)$ is a polarized Dieudonn\'e module as in~\ref{ssec:DM1}.

For the discussion that follows, it is important to recall that we have natural isomorphisms $\Ker(F) \cong H^0(C,\Omega^1_C)$ and $M/\Ker(F) \cong H^1(C,\cO_C)$. The HW-triple associated with~$M$, as in~\ref{ssec:DMtoHW} is of the form $\cQ_C = \bigl(H^1(C,\cO_C),\Phi_C,\Psi_C\bigr)$, where the operator $\Phi_C$ is the classical Hasse--Witt operator on $H^1(C,\cO_C)$, induced by the absolute Frobenius on~$\cO_C$. The operator~$\Psi_C$ may be viewed as an injective $\sigma$-linear map $\Ker(\Phi_C) \to H^0(C,\Omega^1_C)$.

\subsection{}\label{ssec:CINotat}
Let $\mP = \mP^n$ for some $n\geq 2$, and let $C \subset \mP$ be a complete intersection curve, say $C = \cZ(f_1,\ldots,f_{n-1})$ with $f_i$ homogeneous of degree $d_i \geq 2$. Let $d = d_1 + \cdots + d_{n-1}$. If $I$ is a subset of $\{1,\ldots,n-1\}$ we write $f_I = \prod_{i\in I}\, f_i$ and $d_I = \sum_{i\in I}\, d_i$. If $n=2$ we assume $d = d_1 \geq 3$, so that in all cases $d-(n+1) \geq 0$. We also assume that $p \nmid d_i$ for all~$i$.

The following convention will be in force:
\begin{equation}\label{eq:convention}
\text{If a finite set $V \subset \mZ$ is written as $V = \{v_1,\ldots,v_m\}$, it is assumed that $v_1 < v_2 < \cdots < v_m$.}
\end{equation}
(In particular, $m = |V|$.)

We will compute cohomology groups using \v{C}ech cohomology with respect to the standard open cover $\{U_0,\ldots,U_n\}$ of~$\mP$. For $J \subset \{0,\ldots,n\}$, let $U_J = \cap_{j\in J}\, U_j$. If $\cF$ is an abelian sheaf on~$\mP$ and $m\geq 0$, let $C^m(\cF) = \oplus_{|J| = m+1}\, \cF(U_J)$, the direct sum running over all subsets $J \subset \{0,\ldots,n\}$ of $m+1$ elements. The \v{C}ech differentials $C^m(\cF) \to C^{m+1}(\cF)$ are defined as in \cite[\href{https://stacks.math.columbia.edu/tag/01FH}{Tag 01FH}]{stacks-project}. (We use the ``ordered'' version of the \v{C}ech complex.)

Let $\sS(\mP) = \oplus_{r\in \mZ}\, H^0\bigl(\mP,\cO_\mP(r)\bigr) = k[X_0,\ldots,X_n]$ and
\[
\sS(C) = \bigoplus_{r\in \mZ}\, H^0\bigl(C,\cO_C(r)\bigr) = \sS(\mP)/I_C\, ,
\]
where $I_C = (f_1,\ldots,f_{n-1})$. Also define $\sT(\mP) = \oplus_{r\in \mZ}\, H^n\bigl(\mP,\cO_\mP(r)\bigr)$, which is a graded $\sS(\mP)$-module. In what follows we identify~$\sT(\mP)$ with $k[X_0^{\pm 1},\ldots,X_n^{\pm 1}]/L$, where $L \subset k[X_0^{\pm 1},\ldots,X_n^{\pm 1}]$ is the subspace spanned by all monomials $X^e = X_0^{e_0}\cdots X_n^{e_n}$ for which at least one exponent~$e_i$ is non-negative. (A homogeneous form $A \in k[X_0^\pm,\ldots,X_n^\pm]$ of degree~$r$ is sent to the class $[A] \in H^n\big(\mP,\cO_\mP(r)\bigr)$ represented by~$A$ as an element of $C^n\bigl(\cO_\mP(r)\bigr) = \cO_\mP(r)\bigl(U_{01\cdots n}\bigr)$.) Let
\[
\sT(C) = (0:I_C) = \bigl\{t\in \sT(\mP) \bigm| \text{$f_i\cdot t = 0$ for all~$i$} \bigr\}\, ,
\]
which is a graded $\sS(C)$-module.

\subsection{}\label{ssec:Kbullet}
Let $\cE = \oplus_{i=1}^{n-1}\, \cO_\mP(-d_i)$. Consider the Koszul complex
\[
K_\bullet = K_\bullet(f_1,\ldots,f_{n-1}):\qquad \wedge^{n-1} \cE \tto \wedge^{n-2} \cE \tto \cdots \tto \cE \tto \cO_\mP\, ,
\]
with $\wedge^m \cE$ placed in degree~$-m$. We have $\wedge^m \cE \cong \oplus_{|I|=m}\, \cO_\mP(-d_I)$, the sum running over all subsets $I \subset \{1,\ldots,n-1\}$ of $m$ elements. If $I = \{i_1,\ldots,i_m\}$ then on the corresponding direct summand $\cO_\mP(-d_I) \subset \wedge^m \cE$ the Koszul differential $\partial_\Kos \colon \wedge^m \cE \to \wedge^{m-1} \cE$ is given by $s \mapsto \sum_{\nu = 1}^m\, (-1)^{\nu +1}\, f_{i_\nu} \cdot s$. As a particular case of this, note that $\wedge^{n-1} \cE = \det(\cE) = \cO_\mP(-d)$ and $\wedge^{n-2}\cE \cong \oplus_{i=1}^{n-1}\, \cO_\mP(-d+d_i)$; under these identifications $\partial_\Kos \colon \wedge^{n-1} \cE \to \wedge^{n-2} \cE$ is given by $s\mapsto (f_1\cdot s,-f_2\cdot s,\ldots,(-1)^nf_{n-1}\cdot s)$.

The Koszul complex is a resolution of~$\cO_C$. As the sheaves $K_m(r)$ have zero cohomology in degrees $\neq 0, n$, we obtain isomorphisms
\begin{equation}\label{eq:H1COCm}
H^1\bigl(C,\cO_C(r)\bigr) \isomworra H^1(\mP,K_\bullet(r)\bigr) \isomarrow \bigl\{t\in H^n\bigl(\mP,\cO_\mP(r-d)\bigr) \bigr| \text{$f_i\cdot t = 0$ for all~$i$} \bigr\} = \sT(C)_{r-d}\, .
\end{equation}
Let $v_r \colon H^1\bigl(C,\cO_C(r)\bigr) \isomarrow \sT(C)_{r-d}$ be the isomorphism thus obtained. For later use, note that if $g \in \sS(C)_m$ and $\xi \in H^1\bigl(C,\cO_C(r)\bigr)$ then we have $g\cdot v_r(\xi) = v_{r+m}(g\cdot \xi)$.

\subsection{}\label{ssec:shuffsigns}
For the calculations that follow, we need to introduce some signs. The setting for this is that we consider a finite set of integers~$I$ and an integer $i \notin I$. Write $I\cup \{i\} = \{i_1,\ldots,i_m\}$ (with convention~\eqref{eq:convention}), and let $a$ be the index such that $i = i_a$. Then we let
\[
\epsilon(i,I) = (-1)^{a+1}\, .
\]

\subsection{}\label{ssec:KCocyc}
A class in $H^1(\mP,K_\bullet)$ can be represented by a cocycle $\alpha = (\alpha_{n-1},\alpha_{n-2},\ldots,\alpha_0)$ with
\[
\alpha_m \in C^{m+1}(\wedge^m \cE) = \bigoplus_{\substack{J \subset \{0,\ldots,n\}\\ |J| = m+2}}\; \bigoplus_{{\substack{I \subset \{1,\ldots,n-1\}\\ |I| = m}}} \cO_\mP(-d_I)\bigl(U_J\bigr)\, .
\]
Let $\alpha_{m,I,J}$ denote the component of~$\alpha_m$ in $\cO_\mP(-d_I)\bigl(U_J\bigr)$. The assumption that $\alpha$ is a cocycle means that for all subsets $I = \{i_1,\ldots,i_m\} \subset \{1,\ldots,n-1\}$ and $J = \{j_1,\ldots,j_{m+3}\} \subset \{0,\ldots,n\}$ we have
\begin{equation}\label{eq:alphacocyc}
\sum_{\mu=1}^{m+3}\; (-1)^{\mu+1} \alpha_{m,I,J\setminus\{j_\mu\}} = (-1)^{m+1} \sum_{i\notin I}\; \epsilon(i,I) \cdot f_i \cdot \alpha_{m+1,I \cup \{i\},J}\, .
\end{equation}
Under the left isomorphism in~\eqref{eq:H1COCm} (taking $r=0$) the class $[\alpha] \in H^1(\mP,K_\bullet)$ is sent to the class in $H^1(C,\cO_C)$ represented by the cocycle $[\bar\alpha_0]$, where the bar denotes the image under $\pi\colon \cO_\mP \to \cO_C$ (i.e., taking the reduction modulo $(f_1,\ldots,f_{n-1})$). The image of~$[\alpha]$ under the right isomorphism in~\eqref{eq:H1COCm} is the class in $H^n\bigl(\mP,\cO_\mP(-d)\bigr)$ represented by the component~$\alpha_{n-1}$.

\subsection{}\label{ssec:HWCI}
As in \ref{ssec:DMofJ[p]}, let $\cQ_C = \bigl(H^1(C,\cO_C),\Phi_C,\Psi_C\bigr)$ be the Hasse--Witt triple associated with~$C$. The goal of this section is to calculate the equivalence class of this triple.

Define
\[
Q = \sT(C)_{-d}\, ,
\]
and let
\[
v = v_0 \colon H^1(C,\cO_C) \isomarrow Q\, .
\]
be the isomorphism obtained in~\ref{ssec:Kbullet}. Recall that, by definition, $\sT(C)_{-d}$ is a subspace of~$\sT(\mP)_{-d}$, and that the latter is the degree~$-d$ component of $k[X_0^\pm,\ldots,X_n^\pm]/L$. Hence, elements of $H^1(C,\cO_C)$ can be represented as classes~$[A]$ with $A \in k[X_0^\pm,\ldots,X_n^\pm]$ a homogeneous form of degree~$-d$.

\begin{proposition}\label{prop:PhiFormula}
Under the isomorphism~$v$, the operator~$\Phi_C$ on $H^1(C,\cO_C)$ becomes the semi-linear endomorphism~$\Phi$ of the space~$Q$ that is given by
\[
\Phi[A] = \bigl[(f_1\cdots f_{n-1})^{p-1} \cdot A^p\bigr]\, .
\]
\end{proposition}

\begin{proof}
Define an endomorphism~$\phi_m$ of $\wedge^m \cE = \oplus_{|I|=m}\; \cO_\mP(-d_I)$ as a sheaf of abelian groups by the rule that $\phi_m(s) = f_I^{p-1} \cdot s^p$ for $s$ a local section of~$\cO_\mP(-d_I)$. By direct calculation we see that the maps~$\phi_m$ define an endomorphism~$\phi_\bullet$ of the Koszul complex~$K_\bullet$ such that the diagram
\[
\begin{tikzcd}
K_\bullet \ar[r,"\pi"] \ar[d,"\phi_\bullet"]& \cO_C \ar[d,"g\mapsto g^p"]\\
K_\bullet \ar[r,"\pi"] & \cO_C
\end{tikzcd}
\]
(with $\pi \colon K_\bullet \to \cO_C$ the natural map) is commutative. It follows that under the right isomorphism in~\eqref{eq:H1COCm}, $\Phi_C$ corresponds to the endomorphism of $H^1(\mP,K_\bullet)$ induced by~$\phi_\bullet$. Under the left isomorphism in~\eqref{eq:H1COCm} this corresponds to the endomorphism of $\sT(C)_{-d}$ induced by~$\phi_{n-1}$. This gives the assertion.
\end{proof}

\subsection{}\label{ssec:H0Omega1}
We have exact sequences
\begin{equation}\label{eq:EOmOm}
0 \tto \cE \otimes \cO_C(r) \tto \Omega^1_\mP \otimes \cO_C(r) \tto \Omega^1_C(r) \tto 0
\end{equation}
(for $r\in \mZ$) and $0 \tto (\Omega^1_\mP \otimes \cO_C) \tto \cO_C(-1)^{n+1} \tto \cO_C \tto 0$.

The complex $\cE \otimes K_\bullet(r)$ is a resolution of $\cE \otimes \cO_C(r)$. As the terms $\cE \otimes K_m(r)$ have zero cohomology in degrees $\neq 0, n$, we obtain that
\begin{equation}\label{eq:H1EtensorK}
\Ker\Bigl(H^n(\mP,\cE(r-d)) \to H^n\bigl(\mP,\cE \otimes \bigl(\wedge^{n-2}\cE\bigr)(r)\bigr) \Bigr) \isomworra H^1\bigl(\mP,\cE \otimes K_\bullet(r)\bigr) \isomarrow H^1\bigl(C,\cE\otimes \cO_C(r)\bigr)\, .
\end{equation}
The term on the left is $\oplus_{\ell=1}^{n-1} \sT(C)_{r-d-d_\ell}$.

\begin{lemma}
For $r \leq 0$ the boundary map associated with~\eqref{eq:EOmOm} combined with the isomorphisms~\eqref{eq:H1EtensorK} gives an isomorphism
\begin{equation}\label{eq:H0Omega1}
w_r \colon H^0\bigl(C,\Omega^1_C(r)\bigr) \isomarrow \Bigl\{\xi = \bigl(\xi_\ell\bigr)_\ell \in \bigoplus_{\ell=1}^{n-1} \sT(C)_{r-d-d_\ell} \Bigm|
\sum_{\ell=1}^{n-1}\, \tfrac{\partial f_\ell}{\partial X_j}\cdot \xi_\ell = 0 \quad \text{for all $j = 0,\ldots,n$} \Bigr\}\, .
\end{equation}
\end{lemma}

\begin{proof}
For $r \leq 0$ we have $H^0\bigl(C,\Omega^1_\mP \otimes \cO_C(r)\bigr) = 0$, and hence
\[
H^0\bigl(C,\Omega^1_C(r)\bigr) \isomarrow \Ker\Bigl( H^1\bigl(C,\cE \otimes \cO_C(r)\bigr) \tto H^1\bigl(C,\Omega^1_\mP \otimes \cO_C(r)\bigr)\Bigr)\, .
\]
We have a diagram with exact row and column
\[
\begin{tikzcd}
  & H^0\bigl(C,\cO_C(r)\bigr) \ar[d] \ar[dr,"h"] & \\
H^1\bigl(C,\cE \otimes \cO_C(r)\bigr) \ar[r] & H^1\bigl(C,\Omega^1_\mP \otimes \cO_C(r)\bigr) \ar[r] \ar[d] & H^1\bigl(C,\Omega^1_C(r)\bigr) \\
 & H^1\bigl(C,\cO_C(r-1)\bigr)^{n+1} &
\end{tikzcd}
\]
We claim that $h$ is injective. If $r<0$ we have $H^0\bigl(C,\cO_C(r)\bigr) = 0$, so injectivity of~$h$ is trivial. For $r=0$ the map $h \colon H^0(C,\cO_C) \to H^1(C,\Omega^1_C) \cong k$ sends~$1$ to $\prod\, d_i$, which is non-zero due to our assumption that $p\nmid d_i$ for all~$i$.

The injectivity of~$h$ implies that
\[
H^0\bigl(C,\Omega^1_C(r)\bigr) \isomarrow \Ker\Bigl( H^1\bigl(C,\cE \otimes \cO_C(r)\bigr) \tto H^1\bigl(C,\cO_C(r-1)\bigr)^{n+1}\Bigr)\, .
\]
Via the isomorphisms in \eqref{eq:H1COCm} and \eqref{eq:H1EtensorK} this becomes
\[
H^0\bigl(C,\Omega^1_C(r)\bigr) \isomarrow \Ker\Bigl(\; \bigoplus_{\ell=1}^{n-1} \sT(C)_{r-d-d_\ell} \xrightarrow{~\delta~} \sT(C)_{r-1}^{n+1}\Bigr)\, ,
\]
where $\delta(y) = \bigl(\frac{\partial f_\ell}{\partial X_j} \cdot y\bigr)_{j=0,\ldots,n}$ for $y \in \sT(C)_{r-d-d_\ell}$. This gives the assertion.
\end{proof}

\subsection{}
Define
\begin{equation}\label{eq:QprimeDef}
Q^\prime = \Bigl\{\xi = \bigl(\xi_\ell\bigr)_\ell \in \bigoplus_{\ell=1}^{n-1} \sT(C)_{-d-d_\ell} \Bigm|
\sum_{\ell=1}^{n-1}\, \tfrac{\partial f_\ell}{\partial X_j}\cdot \xi_\ell = 0 \quad \text{for all $j = 0,\ldots,n$} \Bigr\}\, ,
\end{equation}
and let
\[
w = w_0\colon H^0(C,\Omega^1_C) \isomarrow Q^\prime\, .
\]

\begin{proposition}\label{prop:PsiFormula}
Under the isomorphisms $v \colon H^1(C,\cO_C) \isomarrow Q$ and $w\colon H^0(C,\Omega^1_C) \isomarrow Q^\prime$, the map $\Psi_C \colon \Ker(\Phi_C) \to H^0(C,\Omega^1_C)$ becomes the map $Q \supset \Ker(\Phi) \to Q^\prime$ given by
\begin{equation}\label{eq:CIPsi}
[A] \mapsto \Bigl[\frac{(f_1\cdots f_{n-1})^{p-1}}{f_1} \cdot A^p\, ,\quad \frac{(f_1\cdots f_{n-1})^{p-1}}{f_2} \cdot  A^p\, ,\quad \ldots \quad ,\quad \frac{(f_1\cdots f_{n-1})^{p-1}}{f_{n-1}} \cdot  A^p  \Bigr]
\end{equation}
\end{proposition}

Note that the right hand side of \eqref{eq:CIPsi} indeed satisfies the required relations; this follows from the remark that
\[
(p-1) \cdot \sum_{\ell=1}^{n-1}\, \tfrac{\partial f_\ell}{\partial X_j}\cdot  \tfrac{(f_1\cdots f_{n-1})^{p-1}}{f_\ell} \cdot A^p = \tfrac{\partial}{\partial X_j} \Bigl((f_1\cdots f_{n-1})^{p-1} A^p \Bigr)
\]
together with the assumption that $\bigl[(f_1\cdots f_{n-1})^{p-1} \cdot A^p\bigr] = 0$ (which says that $[A] \in \Ker(\Phi)$).

\begin{proof}
The proof is based on a \v{C}ech calculation which is elementary but which is notationally a little involved. For the first step of the proof we calculate in the \v{C}ech--Koszul double complex whose term in bi-degree $(-m,r)$ is $C^r(\wedge^m \cE) = \oplus_{|I|=m} \oplus_{|J|=r+1}\; \cO_\mP(-d_I)\bigl(U_J\bigr)$. In the vertical direction we have the \v{C}ech differentials $\partial_\ver \colon C^r(\wedge^m \cE) \to C^{r+1}(\wedge^m \cE)$. In the horizontal direction, the differential $C^r(\wedge^m \cE) \to C^r(\wedge^{m-1} \cE)$ is $(-1)^r\cdot \partial_\Kos$. As in~\ref{ssec:KCocyc}, if $y$ is an element of $C^r(\wedge^m \cE)$ then we denote by $y_{I,J}$ its component in $\cO_\mP(-d_I)\bigl(U_J\bigr)$.

Consider a class $[A] \in \Ker(\Phi)$. Let $\alpha = (\alpha_{n-1},\alpha_{n-2},\ldots,\alpha_0)$ be a cocycle (as in section~\ref{ssec:KCocyc}) that represents the corresponding class in $H^1(\mP,K_\bullet)$. As remarked before, $[A] = [\alpha_{n-1}]$. The assumption that $\Phi[A] = 0$ implies that the cocycle $\bigl(\phi_{n-1}(\alpha_{n-1}),\ldots,\phi_0(\alpha_0)\bigr)$ is a coboundary. Hence there exists an element
\[
\beta = (\beta_{n-1},\ldots,\beta_0)\, ,\qquad \text{with}\quad \beta_m \in C^m(\wedge^m \cE) = \bigoplus_{|J|=m+1}\; \wedge^m\cE(U_J)
\]
such that $\partial_\tot(\beta) = \bigl(\phi_{n-1}(\alpha_{n-1}),\ldots,\phi_0(\alpha_0)\bigr)$. If we write this out we find that for all sets $I = \{i_1,\ldots,i_m\} \subset \{1,\ldots,n-1\}$ and $J = \{j_1,\ldots,j_{m+2}\} \subset \{0,\ldots,n\}$ we have
\begin{equation}\label{eq:alpha=dbeta}
f_I^{p-1} \cdot \alpha_{m,I,J}^p = \sum_{\mu=1}^{m+2}\; (-1)^{\mu+1}\cdot \beta_{m,I,J\setminus \{j_\mu\}} + (-1)^{m+1} \cdot \sum_{i \notin I}\; \epsilon(i,I) \cdot f_i \cdot \beta_{m+1,I\cup\{i\},J}\, .
\end{equation}
For $m=0$ this says that for all subsets $\{j_1,j_2\} \subset \{0,\ldots,n\}$ we have the relation
\begin{equation}\label{eq:casem=0}
\alpha_{0,\emptyset,\{j_1,j_2\}}^p = \bigl(\beta_{0,\emptyset,\{j_2\}} - \beta_{0,\emptyset,\{j_1\}}\bigr) - \sum_{i=1}^{n-1}\; f_i \cdot \beta_{1,\{i\},\{j_1,j_2\}}\, .
\end{equation}
(Note that $\alpha_0 = (\alpha_{0,\emptyset,J})_{|J|=2}$ with $\alpha_{0,\emptyset,J} \in \cO_\mP(U_J)$. Similarly,  $\beta_0 = (\beta_{0,\emptyset,J})_{|J|=1}$ with $\beta_{0,\emptyset,\{j\}} \in \cO_\mP(U_j)$, and $\beta_1 = (\beta_{1,I,J})_{|I|=1, |J|=2}$, with $\beta_{1,\{i\},J} \in \cO_\mP(-d_i)\bigl(U_J\bigr)$.) It follows from \eqref{eq:casem=0} that the $1$-forms $-\dif \bar\beta_{0,\emptyset,\{j\}}$ on $C\cap U_j$ agree on overlaps and therefore give a global $1$-form on~$C$. The $1$-form thus obtained is $\Psi[A]$.

Our aim is to calculate the image of $\Psi[A]$ under the isomorphism~\eqref{eq:H0Omega1}. First we calculate the corresponding element of $\Ker\bigl(H^1(C,\cE \otimes \cO_C) \to H^1(C,\cO_C(-1))^{n+1} \bigr)$, which means we have to compute the image of~$\Psi[A]$ under the boundary map associated with the short exact sequence~\eqref{eq:EOmOm}. Differentiating \eqref{eq:casem=0} and then restricting to~$C$ gives
\[
\bigl(-\dif \beta_{0,\emptyset,\{j_2\}}\bigr) - \bigl(-\dif \beta_{0,\emptyset,\{j_1\}}\bigr) = \sum_{i=1}^{n-1}\; -\beta_{1,\{i\},\{j_1,j_2\}} \cdot \dif f_i\, ,
\]
as sections of $\Omega^1_\mP \otimes \cO_C$. It follows that $\Psi[A]$ is sent to the class in $H^1(C,\cE \otimes \cO_C)$ that is represented by the cocycle $-\bar\beta_1 = (-\bar\beta_{1,I,J})_{|I|=1, |J|=2}$.

Next we will write down a cocycle~$\gamma$ that represents the class in $H^1(\mP,\cE\otimes K_\bullet)$ that under the right isomorphism in~\eqref{eq:H1EtensorK} maps to the class $-[\bar\beta_1]$. Such a cocycle is a tuple $\gamma = (\gamma_{n-1},\ldots,\gamma_0)$ with
\[
\gamma_m \in C^{m+1}\bigl(\cE \otimes \wedge^m \cE\bigr) = \bigoplus_{\ell=1}^{n-1}\; \bigoplus_{\substack{I \subset \{1,\ldots,n-1\\|I|=m}}\; \bigoplus_{\substack{J \subset \{0,\ldots,n\\|J|=m+2}}\; \cO_\mP(-d_\ell-d_I)\bigl(U_J\bigr)\, .
\]
We denote by $\gamma_{m,\ell,I,J}$ the component of~$\gamma$ in $\cO_\mP(-d_\ell-d_I)\bigl(U_J\bigr)$.

We are done if we can show that the tuple~$\gamma$ with
\[
\gamma_{m,\ell,I,J} =
\begin{cases}
(-1)^{m+1} \cdot \epsilon(\ell,I) \cdot \beta_{m+1,I\cup\{l\},J} & \text{if $\ell \notin I$} \\
f_\ell^{p-2} \cdot \Bigl(\prod_{\nu \in I\setminus\{\ell\}}\, f_\nu^{p-1}\Bigr) \cdot \alpha_{m,I,J}^p & \text{if $\ell \in I$}
\end{cases}
\]
is a cocycle. Indeed, it is clear that $\gamma_0$ reduces to $-\bar\beta_1$ modulo $(f_1,\ldots,f_{n-1})$, so that $[\gamma] \in H^1(\mP,\cE\otimes K_\bullet)$ maps to $-[\bar\beta_1]$, which is the image of $\Psi[A]$ in $H^1(C,\cE \otimes \cO_C)$. On the other hand, the image of $[\gamma]$ under the left isomorphism in~~\eqref{eq:H1EtensorK} (taking $r=0$) is the class in $H^n\bigl(\mP,\cE(-d)\bigr)$ represented by the cocycle~$\gamma_{n-1}$, and this gives precisely the result as expressed in~\eqref{eq:CIPsi}.

Choose $\ell \in \{1,\ldots,n-1\}$, a subset $I \subset \{1,\ldots,n-1\}$ with $|I| = m$, and a subset $J = \{j_1,\ldots,j_{m+3}\} \subset \{0,\ldots,n\}$, and consider the component of $\partial_\tot(\gamma)$ in $\cO_\mP(-d_\ell-d_I)\bigl(U_J\bigr)$. To show that $\gamma$ is a cocycle we need to verify that
\[
\sum_{\mu=1}^{m+3}\; (-1)^{\mu+1} \cdot \gamma_{m,\ell,I,J\setminus\{j_\mu\}} \overset{?}{=} (-1)^{m+1} \cdot \sum_{i\notin I}\; \epsilon(i,I)\cdot f_i \cdot \gamma_{m+1,\ell,I\cup\{i\},J}\, .
\]
If $\ell \notin I$ this follows, after some straightforward calculation, from \eqref{eq:alpha=dbeta}; if $\ell \in I$ this follows from \eqref{eq:alphacocyc}.
\end{proof}

\subsection{}\label{ssec:theta}
The final ingredient that we need to describe is a duality between the spaces $Q$ and~$Q^\prime$. Define
\[
\sU = \Bigl\{\xi = \bigl(\xi_\ell\bigr)_\ell \in \bigoplus_{\ell=1}^{n-1} \sT(C)_{n+1-2d-d_\ell} \Bigm|
\sum_{\ell=1}^{n-1}\, \tfrac{\partial f_\ell}{\partial X_j}\cdot \xi_\ell = 0 \quad \text{for all $j = 0,\ldots,n$} \Bigr\}\, .
\]
Choose an isomorphism $\eta \colon \Omega^1_C \isomarrow \cO_C(d-n-1)$ (which is unique up to a constant in~$k^*$). We have isomorphisms $\eta^{-1} \otimes \id \colon \cO_C \isomarrow \Omega^1_C(n+1-d)$ and
\[
H^0(C,\cO_C) \maprightou{\sim}{H^0(\eta^{-1} \otimes \id)} H^0\bigl(C,\Omega^1_C(n+1-d)\bigr) \maprightou{\sim}{w_{n+1-d}} \sU\, .
\]
Let $\sansu \in \sU$ be the image of~$1$, so that $\sU = k\cdot \sansu$. The isomorphism
\[
\sS(C)_{d-n-1} = H^0\bigl(C,\cO_C(d-n-1)\bigr) \maprightou{\sim}{H^0(\eta^{-1})} H^0(C,\Omega^1_C) \maprightou{\sim}{\ w\ } Q^\prime
\]
is then given by $g \mapsto g\cdot \sansu$. Further, we have the isomorphism
\begin{equation}\label{eq:H1Om1}
H^1(C,\Omega^1_C) \maprightou{\sim}{H^1(\eta)} H^1\bigl(C,\cO_C(d-n-1)\bigr) \maprightou{\sim}{v_{d-n-1}} \sT(C)_{-n-1} = k\cdot \bigl[X^{-\mathbf{1}}\bigr]\, ,
\end{equation}
where by $X^{-\mathbf{1}}$ we mean the monomial $X_0^{-1} X_1^{-1} \cdots X_n^{-1}$.

\begin{proposition}
Define a pairing $\langle\ ,\ \rangle \colon Q \times Q^\prime \to \sT(C)_{-n-1}$ by the rule $\bigl\langle [A],g\cdot \sansu\bigr\rangle = g\cdot [A]$, for $g \in \sS(C)_{d-n-1}$. Then the diagram
\[
\begin{tikzcd}
H^1(C,\cO_C) \times H^0(C,\Omega^1_C) \arrow{d}{\wr}[swap]{{(v,w)}} \arrow{r}{\ \cup\ } & H^1(C,\Omega^1_C) \ar{d}{\eqref{eq:H1Om1}}[swap]{\wr} \\
Q \times Q^\prime \arrow{r}{\ \langle\ ,\ \rangle\ } & \sT(C)_{-n-1}
\end{tikzcd}
\]
is commutative.
\end{proposition}

\begin{proof}
Let $\alpha \in H^1(C,\cO_C)$ and $\beta \in H^0(C,\Omega^1_C)$. Let $g = H^0(\eta)\bigl(\beta\bigr) \in \sS(C)_{d-n-1}$, so that $w(\beta) = g\cdot \sansu$. Then $H^1(\eta)\bigl(\alpha \cup \beta\bigr) = \alpha \cup H^0(\eta)\bigl(\beta\bigr) = g \cdot \alpha$ in $H^1\bigl(C,\cO_C(d-n-1)\bigr)$. Hence
\[
v_{d-n-1} \circ H^1(\eta)\bigl(\alpha \cup \beta\bigr) = v_{d-n-1}\bigl(g\cdot \alpha\bigr) = g\cdot v(\alpha) = \bigl\langle v(\alpha),w(\beta) \bigr\rangle\, .
\tag*{\qedhere}\]
\end{proof}

In particular, it follows that $\langle\ ,\ \rangle \colon Q \times Q^\prime \to \sT(C)_{-n-1} \cong k$ is a perfect pairing. Let $\theta \colon Q^\prime \isomarrow Q^\vee$ be the associated isomorphism. Note that the pairing $\langle\ ,\ \rangle$, and hence also the isomorphism~$\theta$, depends on the chosen isomorphism~$\eta$, through the dependence of~$\sansu$ on this choice. By Proposition~\ref{prop:EquivHW} the isomorphism class of the HW-triple that we obtain is independent of how we choose~$\eta$.

We can summarize the results obtained in this section as follows.

\begin{theorem}\label{thm:HWCI}
Let $C \subset \mP^n$ be a complete intersection curve of degrees $(d_1,\ldots,d_{n-1})$ over an algebraically closed field~$k$ of characteristic~$p$ with $p \nmid \prod_{i=1}^{n-1}\; d_i$. With notation as in~\emph{\ref{ssec:CINotat}}, define $Q = \sT(C)_{-d}$. Define $Q^\prime$ as in \eqref{eq:QprimeDef}, and let $\theta \colon Q^\prime \isomarrow Q^\vee$ be the isomorphism defined above. Define a $\sigma$-linear map $\Phi \colon Q \to Q$ by
\[
\Phi[A] = \bigl[(f_1\cdots f_{n-1})^{p-1} \cdot A^p\bigr]\, ,
\]
and define $\Psi \colon \Ker(\Phi) \to Q^\vee$ by
\[
\Psi[A] = \theta\Bigl[\frac{(f_1\cdots f_{n-1})^{p-1}}{f_1} \cdot A^p\, ,\quad \frac{(f_1\cdots f_{n-1})^{p-1}}{f_2} \cdot  A^p\, ,\quad \ldots \quad ,\quad \frac{(f_1\cdots f_{n-1})^{p-1}}{f_{n-1}} \cdot  A^p  \Bigr]\, .
\]
Then $(Q,\Phi,\Psi)$ is a Hasse--Witt triple whose associated Dieudonn\'e module is isomorphic to the Dieudonn\'e module of~$J[p]$, where $J$ is the Jacobian of~$C$.
\end{theorem}

\begin{remark}
If $C$ is a plane curve ($n=2$), the definition of $Q$ and~$Q^\prime$ simplifies; see the theorem in the introduction, in which we express everything using only $\sS = \sS(\mP)$ and $\sT = \sT(\mP)$. Note that in the definition of~$Q^\prime$ as given there, the relations $\frac{\partial f}{\partial X_j} \cdot \xi = 0$ imply that $f \cdot \xi = 0$, by the Euler relation and our assumption that $p\nmid d$. A basis of~$Q$ is in this case given by the classes $X^{-m} \cdot X^{-\mathbf{1}}$ where $m=(m_0,m_1,m_2)$ runs over the elements of~$\mN^3$ with $|m|=d-3$. If $\sansu$ is a non-zero element of~$\sU$ then the dual basis of $Q^\prime = Q^\vee$ (identified via the isomorphism~$\theta$) is given by the elements $X^m \cdot \sansu$.
\end{remark}

\begin{example}\label{exa:C/F5}
Consider the plane curve $C$ over~$\mF_5$ given by $f=0$, where
\[
f = X_0^4 + X_1^4 + X_2^4 + X_0^3 X_1 + X_0 X_1^2 X_2 - X_1^2 X_2^2 + 3\, X_1 X_2^3 \, .
\]
(We omit the verification that $C$ is smooth over~$\mF_5$.) A basis for the space~$Q$ is given by the classes
\[
e_0 = [X_0^{-2} X_1^{-1} X_2^{-1}]\, ,\quad e_1 = [X_0^{-1} X_1^{-2} X_2^{-1}]\, ,\quad e_2 = [X_0^{-1} X_1^{-1} X_2^{-2}]\, .
\]
The Hasse--Witt matrix with respect to this basis is
\[
\begin{pmatrix}
0 & 4 & 1\\
0 & 2 & 3\\
0 & 2 & 3
\end{pmatrix}
\]
(Elements of~$Q$ are represented as column vectors with respect to the chosen basis, and the matrix acts from the left.) In this case, knowing the Hasse--Witt matrix is not sufficient to determine the Ekedahl--Oort type: both
\begin{equation}\label{eq:KraftCycles}
\begin{tikzpicture}
[baseline,auto,knoop/.style={circle,inner sep=0pt,minimum size=8pt,fill,text=white}]
 \node[knoop] (o) at (330:1) {$\scriptstyle \mathbf{2}$};
 \node[knoop] (n) at (0,1) {$\scriptstyle \mathbf{4}$};
 \node[knoop] (w) at (210:1) {$\scriptstyle \mathbf{1}$};
 \path [->,thick,red] (n) edge [bend left=45] node[black] {$F$} (o);
 \path [->,thick,green] (n) edge [bend right=45] node[swap,black] {$V$} (w);
 \path [->,thick,red] (o) edge [bend left=45] node[black] {$F$} (w);
 \begin{scope}[xshift=3cm]
   \node[knoop] (o2) at (330:1)  {$\scriptstyle \mathbf{3}$};
   \node[knoop] (n2) at (0,1) {$\scriptstyle \mathbf{6}$};
   \node[knoop] (w2) at (210:1) {$\scriptstyle \mathbf{5}$};
   \path [->,thick,red] (n2) edge [bend left=45] node[black] {$F$} (o2);
   \path [->,thick,green] (n2) edge [bend right=45] node[swap,black] {$V$} (w2);
   \path [->,thick,green] (w2) edge [bend right=45] node[swap,black] {$V$} (o2);
 \end{scope}
 \node at (1.5,-2) {$s_3s_2 = \left[\begin{smallmatrix} 1 & 2 & 3 & 4 & 5 & 6\\ 1 & 4 & 2 & 5 & 3 & 6\end{smallmatrix}\right]$};
\end{tikzpicture}
\qquad\text{and}\qquad
\begin{tikzpicture}
[baseline,auto,knoop/.style={circle,inner sep=0pt,minimum size=8pt,fill}]
 \node[knoop] (o) at (1,0)  {};
 \node[knoop] (n) at (0,1) {};
 \node[knoop] (w) at (-1,0) {};
 \node[knoop] (z) at (0,-1) {};
 \path [->,thick,red] (n.east) edge [bend left=45] node[black] {$F$} (o.north);
 \path [->,thick,green] (n.west) edge [bend right=45] node[swap,black] {$V$} (w.north);
 \path [->,thick,red] (o.south) edge [bend left=45] node[black] {$F$} (z.east);
 \path [->,thick,green] (w.south) edge [bend right=45] node[swap,black] {$V$} (z.west);

 \node[knoop] (p) at (3,.7) {};
 \node[knoop] (q) at (3,-.7) {};
 \path [->,thick,red] (p) edge [bend left=45] node[black] {$F$} (q);
 \path [->,thick,green] (p) edge [bend right=45] node[swap,black] {$V$} (q);
 \node at (1.25,-2) {$s_3 = \left[\begin{smallmatrix} 1 & 2 & 3 & 4 & 5 & 6\\ 1 & 2 & 4 & 3 & 5 & 6\end{smallmatrix}\right]$};
\end{tikzpicture}
\end{equation}
are still possible. (The pictures represent Dieudonn\'e modules given by their Kraft cycles; see~\cite{Kraft} or \cite{GSAS}, Section~2, or see also the proof of Proposition~\ref{prop:EquivHW}. The numbering of the nodes in the left picture can be ignored for now. The permutations given below the pictures are the minimal representatives of the corresponding elements of $W_X\backslash W$ as in Theorem~\ref{thm:ClassifBT1}. In the moduli space~$\cA_3$, the left picture corresponds to a $2$-dimensional EO-stratum~$\cA_3(s_3s_2)$, the right one to a $1$-dimensional stratum~$\cA_3(s_3)$ which is contained in the boundary of~$\cA_3(s_3s_2)$.)

With the help of a computer (to solve a system of linear equations over~$\mF_5$) we find that $(X_0X_1X_2)^{-7}$ times
\[
\begin{split}
3 X_0^6 X_1^6 + 4 X_0^6 X_1^5 X_2 + 4 X_0^6 X_1^4 X_2^2 + X_0^6 X_1^3 X_2^3 + 2 X_0^6 X_1^2 X_2^4 + 2 X_0^6 X_1 X_2^5 + 3 X_0^6 X_2^6 + 4  X_0^5 X_1^6 X_2 \\
 + 4 X_0^5 X_1^2 X_2^5 + 2 X_0^5 X_1 X_2^6 + 2 X_0^4 X_1^6 X_2^2 + 2 X_0^4 X_1^5 X_2^3 + 4 X_0^4 X_1^4 X_2^4 + 3 X_0^4 X_1^2 X_2^6 + 3 X_0^3 X_1^5 X_2^4 \\
+ 2 X_0^3 X_1^4 X_2^5 + X_0^3 X_1^3 X_2^6 + 4 X_0^2 X_1^6 X_2^4 + X_0^2 X_1^5 X_2^5 + 2 X_0^2 X_1^4 X_2^6 + 2 X_0 X_1^6 X_2^5 + X_0 X_1^5 X_2^6
\end{split}
\]
is a generator $\sansu$ of~$\sU$. (At this stage we find it convenient to ``shift'' everything by a power of $X_0X_1X_2$ so as to obtain polynomial expressions.) The elements
\[
\check{e}_0 = X_0 \cdot \sansu\, ,\quad \check{e}_1 = X_1 \cdot \sansu\, ,\quad \check{e}_2 = X_2 \cdot \sansu
\]
then form a basis of~$Q^\prime$ which is dual (under the isomorphism~$\theta$) to the chosen basis of~$Q$.

The kernel of~$\Phi$ is spanned by $e_0$ and $e_1 + e_2$. We find
\[
\Psi(e_0) = \bigl[f^3 \cdot X_0^{-10} X_1^{-5} X_2^{-5}\bigr] = 3\check{e}_0 + \check{e}_1 + 3\check{e}_2
\]
and
\[
\Psi(e_1+e_2) = \bigl[f^3 \cdot (X_0^{-5} X_1^{-10} X_2^{-5} + X_0^{-5} X_1^{-5} X_2^{-10})\bigr] = 3\check{e}_0 + 3\check{e}_1 + \check{e}_2
\]

In the construction described in Section~\ref{ssec:HWtoDM}, choose $R_0 = k\cdot e_2$ as a complement of $\Ker(\Phi)$, and use $\{e_0,e_1,e_2,\check{e}_2,\check{e}_1,\check{e}_0\}$ as a symplectic basis of $M = Q \oplus Q^\vee$. (Note the order! The standard symplectic form is the one given by the anti-diagonal matrix with coefficients $(-1,-1,-1,1,1,1)$.) With respect to this basis, the matrices of $F$ and~$V$ on~$M$ are given by
\[
F = \begin{pmatrix}
0 & -1 & 1 & 0 & 0 & 0\\
0 & -3 & 3 & 0 & 0 & 0\\
0 & -3 & 3 & 0 & 0 & 0\\
3 & 1 & 0 & 0 & 0 & 0\\
1 & 3 & 0 & 0 & 0 & 0\\
3 & 3 & 0 & 0 & 0 & 0
\end{pmatrix}
\qquad\text{and}\qquad
V = \begin{pmatrix}
0 & 0 & 0 & 0 & 0 & 0\\
0 & 0 & 0 & 0 & 0 & 0\\
0 & 0 & 0 & 0 & 0 & 0\\
0 & 0 & 0 & 3 & 3 & 1\\
-3 & -3 & -1 & -3 & -3 & -1\\
-3 & -1 & -3 & 0 & 0 & 0\\
\end{pmatrix}
\]
(If we write the matrix of~$F$ in block form as $\left(\begin{smallmatrix} A & B \\ C & D\end{smallmatrix}\right)$ and if for a matrix~$M$ we denote by~$M^\dagger$ the matrix obtained by reflection in the anti-diagonal then the matrix of~$V$ is $\left(\begin{smallmatrix} D^\dagger & -B^\dagger \\ -C^\dagger & A^\dagger\end{smallmatrix}\right)$. This means that \eqref{eq:bFxy} holds.) Using the procedure outlined in Section~\ref{ssec:classif} we now find that the $p$-kernel group scheme $J[p]$ is of type $s_3 s_2$ (the left picture in~\eqref{eq:KraftCycles}). To be completely explicit: in $k=\kbar \supset \mF_5$ choose an element~$\alpha$ with $\alpha^{124} = 2$; then the vectors
\[{\scriptstyle
f_1 = \alpha^{-25} \cdot \begin{pmatrix} 0\\ 0\\ 0\\ 2\\ 0\\ -1\end{pmatrix}\, ,\quad
f_2 = \alpha^{-5} \cdot \begin{pmatrix} 2\\ 1\\ 1\\ 3\\ -1\\ -1\end{pmatrix}\, ,\quad
f_3 = \alpha \cdot \begin{pmatrix} 0\\ 0\\ 0\\ 3\\ 2\\ 0 \end{pmatrix}\, ,\quad
f_4 = \alpha^{-1} \cdot \begin{pmatrix} 0\\ 3\\ 0\\ 2\\ 0\\ 0\end{pmatrix}\, ,\quad
f_5 = \alpha^5 \cdot \begin{pmatrix} 0\\ 0\\ 0\\ 0\\ 2\\ 2\end{pmatrix}\, ,\quad
f_6 = \alpha^{25} \cdot \begin{pmatrix} 2\\ 3\\ 3\\ 0\\ 0\\ 0\end{pmatrix}
}\]
form a basis of~$M$ on which $F$ and~$V$ act as in the left diagram in~\eqref{eq:KraftCycles}; i.e., $F(f_4) = f_2$, etc.
\end{example}

\section{A brief description of the Magma implementation}
\label{sec:Magma}

\subsection{}
Jointly with Wieb Bosma (who wrote most of the code), we have implemented our method to calculate the Ekedahl--Oort type of a plane curve in Magma. One of the referees has kindly provided a cleaned-up version of the code, and has given permission to make it available on the author's webpage. On their request we outline how the implementation works.

To use the code, run \texttt{Attach("EOType.txt")}. If a polynomial~$f$ in three variables over a field of characteristic~$p$ has been created, calling \texttt{EOType(f)} then returns the Ekedahl--Oort type of the plane curve defined by~$f$. (It is assumed that $f$ is homogeneous of degree $d\geq 3$ with $p\nmid d$, and that the curve defined by it is smooth over the base field; if not, a message is returned.)

\subsection{Overview.} After checking if the above conditions on $f$ are satisfied, \texttt{EOType} successively calls three functions \texttt{HWtriple}, \texttt{DieudMod}, and \texttt{WeylGrElt}. The function \texttt{HWtriple} implements the theorem stated in the introduction. Next \texttt{DieudMod} converts the Hasse--Witt triple into a Dieudonn\'e module, following the method explained in Section~\ref{sec:DM}. Finally, \texttt{WeylGrElt} computes the Weyl group coset that represents the isomorphism class of the Dieudonn\'e module, under the correspondence given in Theorem~\ref{thm:ClassifBT1}. More details about these functions are given below.

Two general comments: (a) We represent $\sigma$-linear maps by ordinary matrices, choosing bases for the spaces involved. E.g., in the code \texttt{APhi} and~\texttt{APsi} are matrices that represent the maps $\Phi$ and~$\Psi$ that are part of a Hasse--Witt triple. (b) Elements of the space~$\sT$ (notation as in the theorem stated in the introduction) can be represented by Laurent polynomials, but this turns out to be inconvenient because Magma functions such as \texttt{MonomialCoefficient} are only available for polynomials. The solution we use is to multiply Laurent polynomials by a sufficiently high power of $(X_0X_1X_2)$ to ensure that we obtain ordinary polynomials.

\subsection{The function \texttt{HWtriple}.}\label{sec:HWtriple}
This function takes as inputs a field $k$, an integer~$d$ and a polynomial $f$. It outputs a string $s$, a matrix~$A(\Phi)$ of size $g \times g$, where $g = (d-1)(d-2)/2$, a basis $\kappa = \{\kappa_1,\ldots,\kappa_h\}$ of the kernel of the $k$-linear map $k^g \to k^g$ given by~$A(\Phi)$, and a matrix $A(\Psi)$ of size $g \times h$. (The integer $h$ is not known a priori). The purpose of the string~$s$ is only to avoid unnecessary calculations: it will be assigned one of the values \texttt{ordinary}, \texttt{superspecial} or \texttt{interesting}. In the first two cases (which are detected as soon as we have the Hasse--Witt matrix $A(\Phi)$), no further work is required and we can directly output the Ekedahl--Oort type.

A basis for the space~$Q$ is given by the classes of the monomials $m_i^{-1} \cdot (X_0X_1X_2)^{-1}$, where $m_1,\ldots,m_g$ are all monomials in $k[X_0,X_1,X_2]$ of degree $d-3$. These monomials are stored in a sequence called~\texttt{Md}. Next the Hasse--Witt matrix $A(\Phi)$ with respect to this basis is calculated. First we store $F = f^{p-2}$. (For $A(\Phi)$ we need $f^{p-1}$; but we again need $f^{p-2}$ later.) The matrix coefficient $A(\Phi)_{ij}$ is the coefficient of $m_j^p \cdot (X_0X_1X_2)^{(p-1)}$ in $f^{p-1} \cdot m_i = f\cdot F \cdot m_i$. If $A(\Phi)$ is either invertible (ordinary case) or zero (superspecial case), we can stop.

If the curve is neither ordinary nor superspecial, we go on to store a basis $\kappa = \{\kappa_1,\ldots,\kappa_h\}$ of the kernel of the linear map $A(\Phi) \colon k^g \to k^g$. Note that $k^g$ represents the space~$Q$ through the chosen basis of~$Q$, and that a basis of the kernel of the $\sigma$-linear map $\Phi \colon Q \to Q$ is given by the vectors ${}^\tau \kappa_j$, where $\tau$ is the inverse of~$\sigma$. Next we store bases for the spaces $\sT_{-2d+2}$ and $\sT_{-3d+3}$. As explained above, we want to work with ordinary polynomials instead of Laurent polynomials; for this reason, the elements that we use are in fact $(X_0X_1X_2)^{3d-3}$ times a basis. Then we calculate the partial derivatives $\partial f/\partial X_i$ and we compute the matrix \texttt{Multdf} which represents the linear map $\sT_{-3d+3} \to \sT_{-2d+2}^{\oplus 3}$ given by $\xi \mapsto \bigl(\frac{\partial f}{\partial X_0} \cdot \xi,\; \frac{\partial f}{\partial X_1}\cdot \xi,\; \frac{\partial f}{\partial X_2}\cdot \xi\bigr)$. By definition, $\sU$ is the kernel of this map. We choose a generator; but for the same reason as above, what we store is not a generator~$\sansu$ of~$\sU$ but rather $\tilde{\sansu} = (X_0X_1X_2)^{3d-3} \cdot \sansu$, which in the code is called~\texttt{utilde}. The elements $m_i \cdot \sansu$ form a basis of the space $Q^\prime \cong Q^\vee$ which is dual to the chosen basis $\{m_i^{-1} \cdot (X_0X_1X_2)^{-1}\}$ of~$Q$.

The final step of \texttt{HWtriple} is the calculation of the $g\times h$ matrix~$A(\Psi)$. If we write $\kappa_j$ as a column matrix with entries $\kappa^{(j)}_i$ ($i=1,\ldots,g$), the $j$th column of the matrix~$A(\Psi)$ is obtained by solving
\begin{equation}\label{eq:PsiEq}
A(\Psi)_{1j} \cdot m_1 \cdot \sansu + \cdots + A(\Psi)_{gj} \cdot m_g\cdot \sansu = f^{p-2} \cdot \Bigl(\kappa_1^{(j)} \cdot m_1^{-p} \cdot X^{-p\cdot \mathbf{1}} + \cdots + \kappa_g^{(j)} \cdot m_g^{-p} \cdot X^{-p\cdot \mathbf{1}}\Bigr)\, .
\end{equation}
(This is an equation in $\sT_{-2d}$, and $X^{-p\cdot \mathbf{1}}$ means $(X_0X_1X_2)^{-p}$. Note that the $j$th column of the matrix~$A(\Psi)$ is the vector $\Psi({}^\tau\kappa_j)$; as $\Psi \colon \Ker(\Phi) \to Q^\prime$ is given by $[A] \mapsto[f^{p-2} \cdot A^p]$, this leads to equation~\eqref{eq:PsiEq} for the coefficients of $A(\Psi)$.)

Let $c = (2d-1)(2d-2)/2$, which is the number of monomials in $k[X_0,X_1,X_2]$ of degree $2d-3$, and let $M_1,\ldots,M_c$ be those monomials. Because Magma's default is to let matrices act from the right, \eqref{eq:PsiEq} is written as the matrix equation
\begin{equation}\label{eq:APsiEq}
{}^\mathsf{t} A(\Psi) \cdot B = \kappa \cdot C\, ,
\end{equation}
where $B$ and~$C$ are the matrices of size $g \times c$ whose rows express the $m_i \cdot \sansu$ (resp.\ the $f^{p-2} \cdot m_i^{-p} \cdot (X_0X_1X_2)^{-p}$) as vectors with respect to the basis $\{M_j^{-1} \cdot (X_0X_1X_2)^{-1}\}_{j=1,\ldots,c}$ of $\sT_{-2d}$, and where $\kappa$ now is the matrix of size $h\times g$ whose rows give the vectors~$\kappa_j$. Concretely, $B_{ji}$ is the coefficient of $(X_0X_1X_2)^{3d-4}$ in $M_i \cdot m_j \cdot \tilde{\sansu}$, and $C_{ji}$ is the coefficient of $(X_0X_1X_2)^{p-1}\cdot m_j^p$ in $f^{p-2}\cdot M_i$. (Recall that $F = f^{p-2}$ has been calculated before and that we have stored $\tilde{\sansu} = (X_0X_1X_2)^{3d-3} \cdot \sansu$.) Then \eqref{eq:APsiEq} is solved using Magma's function \texttt{IsConsistent}.

\subsection{The function \texttt{DieudMod}.}
\label{sec:DieudMod}
This function takes as input a field $k$, a positive integer $d$, a matrix~$A(\Phi)$ of size $g \times g$, where $g = (d-1)(d-2)/2$, a basis $\{\kappa_1,\ldots,\kappa_h\}$ of the kernel of the $k$-linear map $k^g \to k^g$ given by~$A(\Phi)$, and a matrix $A(\Psi)$ of size $g \times h$. It outputs a matrix~$A(F)$ of size $2g \times g$ whose columns are linearly independent.

We first find a subset $I = \{i_1,\ldots,i_{g-h}\} \subset \{1,\ldots,g\}$ such that $\mathrm{Span}(e_i;i \in I)$ is a complement of $\{\kappa_1,\ldots,\kappa_h\}$ inside~$k^g$. To obtain the $j$th column of the matrix~$A(F)$, write the standard base vector~$e_j$ in the form
\begin{equation}\label{eq:ejDec}
e_j = \sum_{\mu=1}^{g-h}\; a_\mu \cdot e_{i_\mu} + \sum_{\nu=1}^h\; b_\nu \cdot \kappa_\nu\, .
\end{equation}
Then the output is the matrix~$A(F)$ given by
\[
A(F)_{rj} = \begin{cases}
\sum_{\mu=1}^{g-h}\; a_\mu \cdot A(\Phi)_{r,i_\mu} & r=1,\ldots,g \\
\sum_{\nu=1}^h\; b_\nu \cdot A(\Psi)_{2g+1-r,\nu} & r=g+1,\ldots,2g.
\end{cases}
\]

To explain why this is what we want, recall that the goal is to give the matrix of $F \colon M \to M$, where $M = Q \oplus Q^\vee$. (See Section~\ref{ssec:HWtoDM}.) As $F$ factors through the projection $M \to Q$, we only need to give the first $g$ columns. We are identifying $Q$ with~$k^g$ via the basis $\{m_i^{-1} \cdot X^{-\mathbf{1}}\}_{i=1,\ldots,g}$. The dual vector space~$Q^\vee$ is identified with~$Q^\prime$ via the choice of a generator $\sansu \in \sU$, and the dual basis of~$Q^\prime$ is $\{m_i\cdot \sansu\}_{i=1,\ldots,g}$. However, as a preparation for the next step we want to use $e_1,\ldots,e_g,\check{e}_g,\ldots,\check{e}_1$ (note the order!) as a basis for $M = Q \oplus Q^\vee = (k^g) \oplus (k^g)^\vee$, where $e_1,\ldots,e_g$ is the standard basis of~$k^g$.

By our choice of~$I$, the space $R_0 = \mathrm{Span}(e_i)_{i\in I}$ is a complement of ${}^\sigma R_1 := \mathrm{Span}(\kappa_1,\ldots,\kappa_h)$ inside~$k^g$. Then $R_0$ is also a complement of $R_1 = \Ker(\Psi) = \mathrm{Span}({}^\tau\kappa_1,\ldots,{}^\tau\kappa_h)$, where we recall that $\tau = \sigma^{-1}$. With notation as in~\eqref{eq:ejDec}, the decomposition of $e_j$ corresponding to $k^g = R_0 \oplus R_1$ is given by $e_j = \sum_{\mu=1}^{g-h}\; {}^\tau a_\mu \cdot e_{i_\mu} + \sum_{\nu=1}^h\; {}^\tau b_\nu \cdot {}^\tau\kappa_\nu$. This gives the stated formula for the coefficients of the matrix~$A(F)$, where the lower half of each column is put upside down because we order the base vectors of~$(k^g)^\vee$ as $\check{e}_g,\ldots,\check{e}_1$.

\subsection{The function \texttt{WeylGrElt}.}
\label{sec:WeylGrElt}
In this step we calculate (the minimal representative of) the Weyl group coset that, under the bijection in Theorem~\ref{thm:ClassifBT1}, corresponds to our Dieudonn\'e module. As outlined in Section~\ref{ssec:classif}, we have to determine the canonical flag for this. Rather than using the $\sigma^{-1}$-linear Verschiebung, we build this flag using the operations $F$ and~$\perp$ (taking orthogonal complement), as this is more convenient. (The result is the same.) For the conversion to a Weyl group element, we then follow \cite{GSAS}, Section~3.6. In the version of the code provided by the referee, most of the original function \texttt{WeylGrElt} has been moved to a new function called \texttt{EOseq}, and in addition to the Weyl group representative also the corresponding so-called ``final type'' is given.

The function \texttt{WeylGrElt} takes as input a field~$k$, a positive integer $d$, and a matrix $A(F)$ of size $2g \times g$ whose columns are linearly independent. It outputs an element $w \in \gS_{2g}$ (the symmetric group on $2g$ letters), as well as a sequence of integers of length $2g+1$ (the final type). There are several steps that are carried out. In steps (1)--(3) we (partially) fill a table, whose initial state is the following:
\[
\begin{array}{| r || c | c | c | c | c | c | c | c | c |}
\hline
i & 0 & 1 & 2 & \cdots & g-1 & g & g+1 & \cdots & 2g \\
\hline
\mathrm{Basis}(i) & \emptyset & & & & & \text{the $g$ columns of $A(F)$} && &  \{e_1,\ldots,e_{2g}\}\\
\hline
f(i) & 0 & & & & &&& & g \\
\hline
\end{array}
\]
where $\{e_1,\ldots,e_{2g}\}$ denotes the standard basis of~$k^{2g}$. If $\mathrm{Basis}(i)$ is defined, it consists of a set of $i$ linearly independent vectors in~$k^{2g}$, and if $f(i)$ is defined, it is an integer with $0 \leq f(i) \leq i$.  The calculation involves finding the perpendiculars of certain subspaces $W \subset k^{2g}$ with respect to the symplectic form on~$k^{2g}$ that is represented by the block matrix $\bigl({0\atop -\mathsf{J}}{\mathsf{J}\atop 0}\bigr)$, where $\mathsf{J}$ denotes the anti-diagonal matrix of size $g \times g$ with all anti-diagonal entries equal to~$1$. In step~(4) we will then define $f(i)$ for all~$i$, and in step~(5) we convert the sequence~$f$ into a permutation.

In more detail, here is what happens.
\begin{enumerate}
\item Create a table as above.
\item Search for the first index~$i$ such that $\mathrm{Basis}(i)$ is defined, but $f(i)$ is not yet defined. If there is no such~$i$ (in the range $1,\ldots, 2g$), go to step~(4). If $\mathrm{Basis}(i) = \{b_1,\ldots,b_i\}$, calculate the vectors $A(F)\bigl({}^\sigma b_j\bigr)$ ($j=1,\ldots,i$), and let $f(i)$ be the dimension of their $k$-linear span.
\item If $\mathrm{Basis}\bigl(f(i)\bigr)$ is already defined, again do step~(2). If not, do the following: \begin{itemize}
\item Among the vectors $A(F)\bigl({}^\sigma b_1\bigr),\ldots,A(F)\bigl({}^\sigma b_i\bigr)$, find a maximal linearly independent subset, say $\{\beta_1,\ldots,\beta_{f(i)}\}$, and store this collection as $\mathrm{Basis}\bigl(f(i)\bigr)$.
\item Find a basis for the space $\mathrm{Span}\bigl(\beta_1,\ldots,\beta_{f(i)}\bigr)^\perp$ and store this as $\mathrm{Basis}\bigl(2g-f(i)\bigr)$.
\end{itemize}
After this, return to step~(2).
\item If $f(i)$ is defined for all~$i$, go to step~(5). Otherwise, find the first value~$a$ for which $f(a)$ is still undefined, and let $b$ be the next value for which $f(b)$ is defined. It will be true that either $f(a-1) = f(b)$ or that $f(b) = f(a-1) + (b-a+1)$; in the first case, set $f(a), f(a+1), \ldots, f(b-1)$ all equal to $f(a-1)$, in the second case define $f(i)$ for $a\leq i< b$ by the rule $f(i) = f(a-1) + (i-a+1)$. Now repeat this step.
\item Let $j_1 < j_2 < \cdots < j_g$ be the values in $\{1,2,\ldots,2g\}$ with the property that $f(j) = f(j-1)$. (There will be precisely~$g$ such values.) Let $i_1 < i_2 < \cdots < i_g$ be the remaining values. Define a function $w \colon \{1,2,\ldots,2g\} \to \{1,2,\ldots,2g\}$ by $w(j_m) = m$ and $w(i_m) = g+m$. Finally, return the Weyl group element
\[
\left[\begin{matrix}
1 & 2 & \cdots & g & g+1& \cdots & 2g\\
w(1) & w(2) & \cdots & w(g) & w(g+1) & \cdots & w(2g)
\end{matrix} \right]
\]
(given as a product of cycles), and the final type, which is the sequence of numbers~$f(i)$.
\end{enumerate}
\medskip

\noindent
It should be noted that in the actual Magma implementation, the index~$i$ in our table runs from $1$ to $2g+1$, rather than from $0$ to~$2g$; so everything is shifted by~$1$.

{\small

\bigskip

} 

\noindent
\texttt{b.moonen@science.ru.nl}

\noindent
Radboud University Nijmegen, IMAPP, PO Box 9010, 6500GL Nijmegen, The Netherlands

\end{document}